\documentclass[a4paper,11pt]{article}
\usepackage{amsmath}
\usepackage{amssymb}
\usepackage{theorem}
\usepackage{pstricks}
\usepackage{euscript}
\usepackage{epic,eepic}
\usepackage{graphicx}
\usepackage{xcolor}

\PassOptionsToPackage{normalem}{ulem}
\usepackage{amsmath}
\usepackage{amssymb}
\usepackage{theorem}
\usepackage{pstricks}
\usepackage{euscript}
\usepackage{epic,eepic}
\usepackage{graphicx}
\usepackage{setspace}
\PassOptionsToPackage{normalem}{ulem}

\newcommand{\menge}[2]{\big\{{#1} \;|\; {#2}\big\}}

\newcommand{\emp}{\ensuremath{{\varnothing}}}

\newcommand{\scal}[2]{\left\langle{#1}\mid {#2} \right\rangle}

\newcommand{\vuo}{\ensuremath{\mbox{\footnotesize$\square$}}}

\newcommand{\HH}{\ensuremath{\mathcal H}}
\newcommand{\GG}{\ensuremath{\mathcal G}}

\newcommand{\RR}{\ensuremath{\mathbb R}}

\newcommand{\NN}{\ensuremath{\mathbb N}}

\newcommand{\dom}{\ensuremath{\operatorname{dom}}}

\newcommand{\prox}{\ensuremath{\operatorname{prox}}}

\newcommand{\zer}{\ensuremath{\operatorname{zer}}}
\newcommand{\gra}{\ensuremath{\operatorname{gra}}}

\newcommand{\sss}{\ensuremath{\mathsf{s}}}
\newcommand{\ww}{\ensuremath{\mathsf{w}}}
\newcommand{\xx}{\ensuremath{\mathsf{x}}}

\newcommand{\yy}{\ensuremath{\mathsf{y}}}

\newcommand{\hh}{\ensuremath{\boldsymbol{h}}}

\newcommand{\rr}{r}

\newcommand{\E}{\ensuremath{\mathsf{E}}}

\newcommand{\AAA}{A}

\newcommand{\Id}{\ensuremath{\operatorname{Id}}}

\newcommand{\pinf}{\ensuremath{+\infty}}

\newtheorem{theorem}{Theorem}[section]
\newtheorem{lemma}[theorem]{Lemma}

\newtheorem{corollary}[theorem]{Corollary}

\newtheorem{definition}[theorem]{Definition}
\theoremstyle{plain}{\theorembodyfont{\rmfamily}
\newtheorem{assumption}[theorem]{Assumption}}
\theoremstyle{plain}{\theorembodyfont{\rmfamily}
}
\theoremstyle{plain}{\theorembodyfont{\rmfamily}
\newtheorem{algorithm}[theorem]{Algorithm}}
\theoremstyle{plain}{\theorembodyfont{\rmfamily}
}
\theoremstyle{plain}{\theorembodyfont{\rmfamily}
\newtheorem{problem}[theorem]{Problem}}
\theoremstyle{plain}{\theorembodyfont{\rmfamily}
\newtheorem{remark}[theorem]{Remark}}
\theoremstyle{plain}{\theorembodyfont{\rmfamily}
}

\definecolor{labelkey}{rgb}{0,0.08,0.45}
\definecolor{refkey}{rgb}{0,0.6,0.0}
\definecolor{Brown}{rgb}{0.45,0.0,0.05}
\definecolor{dgreen}{rgb}{0.00,0.49,0.00}
\definecolor{dblue}{rgb}{0,0.08,0.75}
\numberwithin{equation}{section}

\title{A Tseng type stochastic forward-backward algorithm for monotone inclusions}

\author{
Van Dung Nguyen\textsuperscript{a} and Nguyen The Vinh\textsuperscript{a}\thanks{CONTACT N. T. Vinh. Email: thevinhbn@utc.edu.vn}\\
\\
$^a$Department of Mathematical Analysis,\\ University of Transport and Communications, 3 Cau Giay Street, Hanoi, Vietnam \\
dungnv@utc.edu.vn;\ thevinhbn@utc.edu.vn
}
\begin{document}
\maketitle
\begin{abstract}
In this paper, we propose a stochastic version of the classical Tseng's forward-backward-forward method with inertial term for solving monotone inclusions given by the sum of a maximal monotone operator and a single-valued monotone operator in real Hilbert spaces. We obtain the almost sure convergence for the general case and the rate $\mathcal{O}(1/n)$ in expectation for the strong monotone case. 
    Furthermore, we derive $\mathcal{O}(1/n)$ rate convergence of the primal-dual gap for saddle point problems. 
\end{abstract}

\noindent {\bf Mathematics Subject Classifications (2010)}:

\section{Introduction} \label{sec:introduction}
In this paper, we study the following inclusion problem:
\begin{align}\label{pt1}
\text{find $x^*\in \mathcal{H}$ such that}\ \ 0\in Ax^*+Bx^*,
\end{align}
where $\mathcal{H}$ is a separable real Hilbert space, $A: \mathcal{H}\to 2^\mathcal{H}$ is a maximal monotone operator and $B:\mathcal{H}\to \mathcal{H}$ is a monotone operator. The solution set of (\ref{pt1}) is denoted by $(A+B)^{-1}(0)$.

This problem 
plays an important role in many fields,
such as equilibrium problems, fixed point problems, variational inequalities, and composite minimization problems, see, for example, \cite{Aoyama,brezis,Combet2017}. To be more precise, many
problems in signal processing,
computer vision and machine learning
can be modeled mathematically as this formulation, see \cite{Combet1,peyre,beck} and the references therein. For solving the problem (\ref{pt1}), the so-called forward-backward splitting method
is given as follows:
\begin{align}\label{pt2}
x_{n+1}=(I+\lambda A)^{-1}(x_n-\lambda Bx_n),
\end{align}
where $\lambda>0$.

The forward-backward splitting algorithm for  monotone inclusion problems was first introduced by Lions and Mercier \cite{lion}. In the work of Lions and Mercier, other splitting methods, such as Peaceman–
Rachford algorithm \cite{peaceman} and Douglas-Rachford algorithm \cite{douglas} was developed to find the zeros of the sum of two maximal monotone operators.
Since then, it has been studied and reported extensively in the literature; see, for instance, \cite{tseng,Vol,combet,bot,prasit2,eckstein,thong} and the references therein. 
Recently, stochastic versions
of splitting algorithms for monotone inclusions have been proposed, for example stochastic forward-backward splitting method \cite{bang1,bang2},
stochastic Douglas-Rachford splitting method \cite{bang3},  stochastic reflected forward-backward splitting method \cite{DB1} and stochastic primal-dual method \cite{bang4}, see also \cite{Combet1,bianchi,pesquet} and applications to stochastic optimization \cite{bang1,Combettes2015} and
machine learning \cite{bang5,duchi}.

 Motivated and inspired by the algorithms in \cite{tseng,bang1,bang2,bang6,DB1}, we will introduce a new stochastic splitting algorithm for inclusion problems. The convergence and the rate convergence of the proposed algorithm are obtained.
 
The rest of the paper is organized as follows. After collecting some definitions and basic results in Section 2, we prove in Section 3 the almost sure convergen for the general case and the strong convergence along with the rate convergence in the strongly monotone case. 

In section 4, we apply the proposed algorithm to the convex-concave saddle point problem. 
\section{Preliminaries}\label{sec2}
\label{Sec:Prelims}
Let $\mathcal{H}$ be a separable real Hilbert space endowed with the inner product $\scal{}{}$ and the associated norm $\|\cdot\|$. When $(x_n)_{n \in \NN}$ is a sequence in $\mathcal{H}$, 
we denote strong convergence of $(x_n)$ to $x\in \mathcal{H}$ by $x_n\to x$ and weak convergence by $x_n\rightharpoonup x$.

We recall some well-known definitions.
\begin{definition} Let $A : \mathcal{H} \to 2^\mathcal{H}$ be a set-valued mapping with
	nonempty values.
	\begin{enumerate}
		\item[\textup{(1)}] $A$ is said
		to be monotone if for all $x, y \in \mathcal{H},\ u \in Ax$
		and $v \in Ay$, the following inequality holds $\scal{ u - v}{ x -y} \geq 0$. 
		\item[\textup{(2)}]$A$ is said
		to be
		{maximally monotone}, if it is monotone and if for any $(x, u) \in \mathcal{H} \times \mathcal{H}$, $\scal{u - v}{ x - y} \geq 0 $ for every $(y, v) \in$ $\gra A=\{(x,y):y\in Ax\}$ (the graph of mapping $A$) implies that $u \in Ax$.
		\item[\textup{(3)}] We say that  $A$  is $\phi_A$-uniformly monotone, if there 
 exists an increasing function $\phi_A\colon\left[0,\infty\right[\to \left[0,\infty\right]$ that vanishes only at $0$ such that 
 \begin{equation}
\big(\forall (x,u),(y,v)\in\gra A\big)
\quad\scal{x-y}{u-v}\geq \phi_A(\|y-x\|).
\end{equation}
If $\phi_A = \nu_A|\cdot|^2$  for some $\nu_A\in \left]0,\infty\right[$, then we say that $A$ is $\nu_A$-strongly monotone.
		\item[\textup{(4)}] Let $\Id$ denote the identity operator on $\mathcal{H}$ and
		$A : \mathcal{H} \to 2^\mathcal{H}$ be a maximal monotone operator. For each $\lambda>0$, {the resolvent mapping} $J^A_\lambda:\mathcal{H}\to \mathcal{H}$ associated with $A$ is defined by
		\begin{equation}
		J^A_\lambda(x):=(\Id+\lambda A)^{-1}(x) \  \forall x\in \mathcal{H}.
		\end{equation}
	
	\end{enumerate}
\end{definition}
\begin{definition} A mapping $T:\mathcal{H}\to \mathcal{H}$ is said to be
	\begin{enumerate}
		\item[\textup{(1)}] {firmly nonexpansive} if
		\begin{equation*}
		\|Tx-Ty\|^2\le \scal{Tx-Ty}{x-y} \ \ \  \forall x,y\in \mathcal{H},
		\end{equation*}
		or equivalently
		\begin{equation*}
		\|Tx-Ty\|^2\le\|x-y\|^2-\|(I-T)x-(I-T)y\|^2  \ \ \  \forall x,y\in \mathcal{H}.
		\end{equation*}
		\item[\textup{(2)}] $L$-{Lipschitz continuous} with $L>0$ if
		\begin{equation}
		\|Tx-Ty\|\le L \|x-y\|  \ \ \ \forall x,y \in \mathcal{H}.
		\end{equation}
	\end{enumerate}
\end{definition}
Let $\Gamma_0(\mathcal{H})$ be the class of proper lower semicontinuous convex functions from $\mathcal{H}$ to $\left ]-\infty,+\infty\right]$. 
\begin{definition}
For $f \in \Gamma_0(\mathcal{H})$:
\begin{enumerate}
\item[\textup{(1)}] $ \dom f=\{x \in \HH,\ f(x)<+\infty \}$. The subdifferential of $f$ at $x \in \dom f$ is $$ \partial f(x)=\{u \in \HH,\ \forall z \in \dom f:\ f(z) \ge f(x)+\scal{u}{z-x} \}.$$
\item[\textup{(2)}] The proximity operator of $f$ is
\begin{align}
\text{prox}_f:\mathcal{H} \to \mathcal{H}: x \mapsto \underset{y \in \mathcal{H}}{\operatorname{argmin}}\big(f(y)+\frac{1}{2}\|x-y\|^2\big).
\end{align}
\item[\textup{(3)}] The conjugate function of $f$ is
\begin{align}
f^*:a \mapsto \underset{x \in \mathcal{H}}{\sup} \big( \scal{a}{x}-f(x) \big).
\end{align}
 \item[\textup{(4)}] The infimal convolution of the two functions $\ell$ and $g$ from $\HH$ to $\left]-\infty,+\infty\right]$ is 
 \begin{equation}
 \ell\;\vuo\; g\colon x \mapsto \inf_{y\in\HH}(\ell(y)+g(x-y)).
\end{equation}
 \end{enumerate}
\end{definition}
Note that  $\text{prox}_f=J_{\partial f}$ and  
\begin{align}
(\forall f \in \Gamma_0(\mathcal{H})) \ \ (\partial f)^{-1}=\partial f^*.
\end{align}

We now recall some results which are needed in sequel.

\begin{lemma}\label{lv1} {\rm(\cite{Ta2010})} Let $A: \mathcal{H}\to 2^\mathcal{H}$ be a set-valued maximal monotone mapping and $\lambda>0$. Then the domain of the resolvent of $A$ is the whole space, that is $D(J^A_\lambda)=\mathcal{H}$, and in addition $J^{A}_\lambda$ is a single-valued and firmly nonexpansive mapping.
\end{lemma}

\begin{lemma}\label{2lv1} {\rm(\cite{Br1973}, Lemma 2.4)}
	Let $A: \mathcal{H} \to 2^\mathcal{H}$ be a maximal monotone mapping and  $B: \mathcal{H} \to \mathcal{H}$
	be a Lipschitz continuous and monotone mapping. Then the mapping  $A +B$ is a maximal monotone mapping.
\end{lemma}
\indent 
Following \cite{Led}, let $(\Omega, {\EuScript{F}},\mathsf{P})$ be a probability space. A $\mathcal H$-valued random variable is a measurable function $X:\Omega \to \mathcal H$, where $\mathcal H$ is endowed with the Borel $\sigma$-algebra. We denote by $\sigma(X)$ the $\sigma$-field generated by $X$. The expectation of a random variable $X$ is denoted by $\mathsf{E}[X]$. The conditional expectation of $X$ given a $\sigma$-field ${\EuScript{A}} \subset {\EuScript{F}}$ is denoted by $\mathsf{E}[X|{\EuScript{A}}]$. A $\mathcal H$-valued random process is a sequence $(x_n)$ of $\mathcal H$-valued random variables. The abbreviation a.s. stands for 'almost surely'.
\begin{lemma}{\upshape(\cite[Theorem 1]{robbins})} \label{lm1}
	Let $({\EuScript{F}}_n)_{n \in \mathbb N}$ be an increasing sequence of sub-$\sigma$-algebras of ${\EuScript{F}}$, let $(z_n)_{n \in \NN},\ (\beta_n)_{n \in \NN},\ (\theta_n)_{n \in \NN}$ and $(\gamma_n)_{n \in \NN}$ be $[0,+\infty]$-valued random sequences such that, for every $n \in \mathbb N$, $z_n,\ \beta_n,\ \theta_n$ and $\gamma_n$ are ${\EuScript{F}}_n$-measurable. Suppose that $\sum_{n \in \mathbb N}\gamma_n<+\infty,\ \sum_{n \in \mathbb N}\beta_n<+\infty \ a.s.$ and
	\begin{align*}
	(\forall n \in \mathbb N)\ \mathsf{E}[z_{n+1}|{\EuScript{F}}_n] \le (1+\gamma_n)z_n+\beta_n-\theta_n \ a.s..
	\end{align*}
	Then $z_n$ converges a.s. and $(\theta_n)_{n \in \NN}$ is summable a.s..
\end{lemma}
According to the proof of Proposition 2.3 
\cite{Combettes2015}, we have the following lemma.
\begin{lemma} \label{lm2}
	Let $C$ be a non-empty closed subset of $\mathcal{H}$ and let $(x_n)_{n\in\NN}$ be a $\mathcal{H}$-valued random process.
	Suppose that, for every $x\in C$, $(\|x_{n+1}-x\|)_{n\in\NN}$ converges a.s.. 
Suppose that the set of weak sequentially cluster points of $(x_n)_{n\in\NN}$ is a subset of C a.s.. Then $(x_n)_{n\in\NN}$ converges weakly a.s. to a $C$-valued random vector.
\end{lemma}

\section{Main results}\label{sec:Alg}
In this section, we propose a novel stochastic forward-backward-forward algorithm for solving the problem \ref{pt1} and analyse its convergence behaviour. Unless otherwise specified, we assume that the following assumptions are satisfied from now on.
\begin{assumption}\label{giathiet}\upshape  In what follows we suppose the following assumptions for $A$ and $B$:
	\begin{enumerate}
		
		\item[\textup{(A1)}] The mapping $B:\mathcal{H}\to \mathcal{H}$ is $L$-Lipschitz continuous and monotone;
		\item[\textup{(A2)}] The set-valued maping $A: \mathcal{H}\to 2^\mathcal{H}$ is maximal monotone.
		\item[\textup{(A3)}] The solution set $\mathcal{P}=\zer(A+B)=(A+B)^{-1}(0)\ne \emptyset$.
		
	\end{enumerate}
\end{assumption}
The algorithm is designed as follows.
\begin{algorithm}
	\label{alg1}
	\vspace*{0.3em}
	{\bf Step 0:} {(Initialization)} Choose $\theta \in [0,1]$. Let $(\lambda_n)_{n \in \NN}$ be a positive sequence, $(\epsilon_n)_{n \in \NN}\subset[0,+\infty)$ satisfying
		\begin{align} \sum_{n=0}^{+\infty}\epsilon_n<+\infty. \\
	 	\end{align}
	Let $x_{-1},\ x_0$ be  $\mathcal{H}$-valued, squared integrable
	random variables and set $n=0$.\\ \vspace*{0.2cm}\noindent
	{\bf Step 1:} Given $x_{n-1},\ x_n$ {\upshape(}$n\geq 0${\upshape)}, choose $\alpha_n$ such that
	\begin{align}\label{dka1}
	{\alpha}_n=\begin{cases}
	\min\bigg\{\dfrac{\epsilon_n}{\|x_n-x_{n-1}\|},\theta\bigg\}&\text{if $x_n \neq x_{n-1}$,}\\
	\theta&\text{if $x_n=x_{n-1}$.}
	\end{cases}
	\end{align}
	Let $r_n$ be a random vector.
Compute
	\begin{align*}
	w_n&=x_n+\alpha_n(x_n-x_{n-1}),\\
	y_n&=(I+\lambda_n A)^{-1}(w_n-\lambda_n r_n).
	\end{align*}
	{\bf Step 2:} Let $s_n$ be an unbiased estimator of $By_n$, i.e., $\mathsf{E}[s_n|{\EuScript{F}_n}]=By_n$. Calculate the next iterate as
	\begin{align}\label{3.9}
	x_{n+1}&=y_n-\lambda_n (s_n-r_n),
	\end{align}
	where $\EuScript{F}_n=\sigma(x_{-1},x_0,r_0,x_1,r_1,\ldots,x_n,r_n)$.\\
	Let $n:=n+1$ and return to {\bf Step 1}.
	
\end{algorithm}
\begin{remark}\upshape
	Some remarks on the algorithm are in order now. 
	\begin{enumerate}
	\item[{\upshape(1)}] Algorithm \ref{alg1} is an extension of the forward-backward-forward splitting method in \cite{tseng} which is in the deterministic setting. In the setting of this method, we do not need the cocoercive condition as in \cite{Combettes2015,Combettes2016,bang2}. 
		\item [{\upshape(2)}]When $\alpha_n=0$, Algorithm \ref{alg1} reduces to (3.2) in \cite{bang6}. However, the conditions for the convergences are different from that in \cite{bang6}.
		\item [{\upshape(3)}]In \cite{cui1,cui2}, for solving \eqref{pt1}, the authors designed stochastic forward-backward-forward splitting methods which require a large number of samples in each iteration. Our results are also different from that in \cite{cui1,cui2}.
			\item [{\upshape(4)}] Evidently, we have from (\ref{dka1}) that
		\begin{align}\label{hqal}
		 \alpha_n\|x_n-x_{n-1}\|\leq\epsilon_n.
		 \end{align}
		
	\end{enumerate}
\end{remark}

\begin{lemma} \label{lm1add}
 Let $(x_n)$ be generated by Algorithm \ref{alg1}, then the following holds:   
 \begin{align}
     \|x_{n+1}-p\|^2 &\le \|w_n-p\|^2-(1- \lambda_n ^2 L^2) \|w_n-y_n\|^2+ \lambda_n^2 \big( \|s_n-By_n\|^2+\|r_n-Bw_n\|^2\big)\notag \\
&\hskip2cm\ +2\lambda_n^2 \big(\scal{s_n-By_n}{By_n-r_n}+\scal{By_n-Bw_n}{Bw_n-r_n} \big) \notag \\
&\hskip2cm\ +2 \scal{y_n-w_n-\lambda_n (By_n-r_n)}{y_n-p}  +2\lambda_n \scal{By_n-s_n}{y_n-p},\label{elm1}
 \end{align}
 for any $p\in \mathcal{P}$.
\end{lemma}
\begin{proof}We have
    	\begin{align}\label{eq15}
	\|x_{n+1}-p\|^2&=\|y_n-\lambda_n(s_n-r_n)-p\|^2 \notag \\
	&=\|y_n-w_n-\lambda_n(s_n-r_n)+w_n-p\|^2\notag \\
	&=\|w_n-p\|^2+\|y_n-w_n-\lambda_n(s_n-r_n)\|^2+2 \scal{y_n-w_n-\lambda_n(s_n-r_n)}{w_n-p} \notag \\
	&=\|w_n-p\|^2+\|y_n-w_n-\lambda_n(s_n-r_n)\|^2+2\scal{y_n-w_n-\lambda_n(s_n-r_n)}{w_n-y_n} \notag \\
	&\hskip7.2cm+2\scal{y_n-w_n-\lambda_n(s_n-r_n)}{y_n-p} \notag \\
	&=\|w_n-p\|^2+\|y_n-w_n-\lambda_n(s_n-r_n)\|^2-2\|w_n-y_n\|^2+2\lambda_n \scal{s_n-r_n}{y_n-w_n} \notag \\
	&\hskip7.2cm+2\scal{y_n-w_n-\lambda_n(s_n-r_n)}{y_n-p}  \notag\\
	&=\|w_n-p\|^2-\|w_n-y_n\|^2+\lambda_n^2\|s_n-r_n\|^2\notag\\
	&\hskip1.5cm +2 \scal{y_n-w_n-\lambda_n (By_n-r_n)}{y_n-p} +2\lambda_n \scal{By_n-s_n}{y_n-p}.
	\end{align}
	Note that
	\begin{align}\label{eq16}
	    \|s_n-r_n\|^2&=\|s_n-By_n+By_n-r_n\|^2 \notag \\
	    &=\|s_n-By_n\|^2+2\scal{s_n-By_n}{By_n-r_n} +\|By_n-Bw_n+Bw_n-r_n\|^2 \notag \\
	    &=\|s_n-By_n\|^2+2\scal{s_n-By_n}{By_n-r_n}+\|By_n-Bw_n\|^2+\|Bw_n-r_n\|^2\notag \\
	    &\ \ \ +2\scal{By_n-Bw_n}{Bw_n-r_n} \notag\\
	    &\le \|s_n-By_n\|^2+2\scal{s_n-By_n}{By_n-r_n}+L^2\|y_n-w_n\|^2+\|Bw_n-r_n\|^2\notag \\
	    &\ \ \ +2\scal{By_n-Bw_n}{Bw_n-r_n}
	\end{align}
By combining (\ref{eq15}) and (\ref{eq16}) we obtain
\begin{align*}
\|x_{n+1}-p\|^2 &\le \|w_n-p\|^2-(1- \lambda_n ^2 L^2) \|w_n-y_n\|^2+ \lambda_n^2 \big( \|s_n-By_n\|^2+\|r_n-Bw_n\|^2\big)\notag \\
&\hskip2cm\ +2\lambda_n^2 \big(\scal{s_n-By_n}{By_n-r_n}+\scal{By_n-Bw_n}{Bw_n-r_n} \big) \notag \\
&\hskip2cm\ +2 \scal{y_n-w_n-\lambda_n (By_n-r_n)}{y_n-p} +2\lambda_n \scal{By_n-s_n}{y_n-p}.
\end{align*}

The proof is complete.
\end{proof}
\begin{theorem} \label{thm1}
 Let $(x_n)_{n \in \NN}$ be generated by Algorithm \ref{alg1}. The followings hold
 \begin{enumerate}
\item \label{tr1i}
Assume that $(\lambda_n)_{n \in \NN}$ be a sequence in $\left]\epsilon, \dfrac{1-\epsilon}{L}\right[$ and the following conditions are satisfied for $\EuScript{F}_n=\sigma(x_{-1},x_0,r_0,x_1,r_1,\ldots,x_n,r_n)$
\begin{equation}\label{e:sum}
 \sum_{n\in\mathbb{N}}\mathsf{E}[\|s_n-By_n\|^2|\EuScript{F}_n]<
 +\infty\ \  a.s..\ \ \text{and}\ \ \sum_{n\in\mathbb{N}}\|r_n-Bw_n\|^2<+\infty \ \ a.s.
 \end{equation}
 Then  $(x_n)_{n \in \NN}$ converges weakly to a random varibale $\overline{x}\colon\Omega\to\zer(A+B)$ a.s..
 \item\label{tr1ii} Suppose that $A$ or $B$ is uniformly monotone. Let $(\lambda_n)_{n \in \NN}$ be a sequence in $\left]0, +\infty \right[$ such that $ (\lambda_n)_{n \in \NN}\in\ell_{2}(\NN)\backslash\ell_1(\NN)$ and
 \begin{equation}
     \sum_{n \in \NN}\lambda^{2}_n \|r_n-Bw_n\|^2 < \infty\ \ a.s\ \  \text{and}\ \ \sum_{n\in\mathbb{N}}\lambda_n^2\mathsf{E}[\|s_n-By_n\|^2|\EuScript{F}_n]<
 +\infty\ \   \text{a.s}., \label{dkv}
 \end{equation}
 where $(\forall p\in \left]0,\infty\right[)\;\ell_p(\mathbb N)=\menge{(\lambda_n)_{n\in\NN}}{(\forall n\in\NN)\; \lambda_n\in\RR, \sum \limits_{n \in \mathbb N} |\lambda_n|^p<+\infty }.$
 Then $(x_n)_{n \in \NN}$ converges strongly to a unique solution $\overline{x}$ a.s.. 
 \end{enumerate}
\end{theorem}
\begin{proof}
From Lemma \ref{lm1add}, taking conditional expectation given $\EuScript{F}_n$ on both sides of \eqref{elm1}), using $\mathsf{E}[s_n|{\EuScript{F}_n}]=By_n$ we get
\begin{align}
     \mathsf{E}[\|x_{n+1}-p\|^2|\EuScript{F}_n] &\le \|w_n-p\|^2-(1- \lambda_n ^2 L^2) \|w_n-y_n\|^2+\lambda_n^2 \E[\|s_n-By_n\|^2|\EuScript F_n]+\lambda_n^2\|r_n-Bw_n\|^2 \notag \\
     &\ +2 \lambda_n^2 \scal{By_n-Bw_n}{Bw_n-r_n} + 2\scal{y_n-w_n-\lambda_n (By_n-r_n)}{y_n-p}.
      \label{ethr1add}
\end{align}
     Since $y_n=(I+\lambda A)^{-1}(w_n-\lambda_n r_n)$, we obtain
     $$\dfrac{w_n-y_n}{\lambda_n}-r_n \in Ay_n$$
     which is equivalent to
     $$\dfrac{w_n-y_n}{\lambda_n}-r_n+By_n \in (A+B)y_n.$$
     We have $0 \in (A+B)p$, using the uniformly monotone of $A+B$, we get
     \begin{align}
        \scal{\dfrac{w_n-y_n}{\lambda_n}-r_n+By_n}{y_n-p} \ge  \phi(\|y_n-p\|),
     \end{align}
     which implies 
     \begin{align}
         \scal{y_n-w_n-\lambda_n (By_n-r_n)}{ y_n-p}  \le -\lambda_n  \phi(\|y_n-p\|).\label{inen2}
     \end{align}
     	   Using \eqref{hqal} and Cauchy-Schwarz inequality, we estimate the term $\|w_n-p\|^2$ in (\ref{ethr1add})  as follows:
	  	\begin{align}
	\|w_n-p\|^2&=\|x_n+\alpha_n(x_n-x_{n-1})-p\|^2\notag\\ 
	& =\|x_n-p\|^2+2\alpha_n \scal{x_n-p}{x_n-x_{n-1}}  +\alpha_n^2 \|x_n-x_{n-1}\|^2\notag\\
	& \le \|x_n-p\|^2+2\alpha_n \|x_n-p\|\|x_n-x_{n-1}\|+\alpha_n^2\|x_n-x_{n-1}\|^2\notag\\
	&\le \|x_n-p\|^2+2\epsilon_n\|x_n-p\|+\epsilon_n^2\notag\\
	&\le (1+\epsilon_n)\|x_n-p\|^2+\epsilon_n^2+\epsilon_n \ .\label{eq21}
	\end{align}
	Therefore, from \eqref{ethr1add}, using \eqref{inen2} and \eqref{eq21}, we derive
	\begin{align}\label{eq20}
	  \mathsf{E}[\|x_{n+1}&-p\|^2|\EuScript{F}_n] \le (1+\epsilon_n)\|x_n-p\|^2-(1- \lambda_n ^2 L^2) \|w_n-y_n\|^2+ \lambda_n^2 \E[\|s_n-By_n\|^2|\EuScript F_n]\notag \\
	  & +\lambda_n^2\|r_n-Bw_n\|^2 +2 \lambda_n^2 \scal{By_n-Bw_n}{Bw_n-r_n}-2 \lambda_n\phi(\|y_n-p\|)+\epsilon_n^2+\epsilon_n.
	  	  \end{align}
	  	  (i) In general case, i.e. $\phi=0$.   We have 
\begin{align}
     2\scal{By_n-Bw_n}{Bw_n-r_n}&\le 2\|By_n-Bw_n\|\|Bw_n-r_n\| \notag \\
     &\le \dfrac{\epsilon}{1-\epsilon}\|By_n-Bw_n\|^2+\dfrac{1-\epsilon}{\epsilon}\|r_n-Bw_n\|^2\notag \\
     &\le \dfrac{\epsilon}{1-\epsilon}L^2\|y_n-w_n\|^2+\dfrac{1-\epsilon}{\epsilon}\|r_n-Bw_n\|^2. \label{inen1}
\end{align}
Hence, \eqref{eq20} implies that
	  	\begin{align}\label{eq201}
	  \mathsf{E}[\|x_{n+1}-p\|^2|\EuScript{F}_n] &\le (1+\epsilon_n)\|x_n-p\|^2-(1- \lambda_n ^2 L^2(1+\frac{\epsilon}{1-\epsilon}) \|w_n-y_n\|^2+ \lambda_n^2 \E[\|s_n-By_n\|^2|\EuScript F_n]\notag \\
	  &\ \ \ +\lambda_n^2(1+\frac{1-\epsilon}{\epsilon})\|r_n-Bw_n\|^2+\epsilon_n^2+\epsilon_n \notag \\
	  &\le (1+\epsilon_n)\|x_n-p\|^2-\epsilon \|w_n-y_n\|^2+ \lambda_n^2 \E[\|s_n-By_n\|^2|\EuScript F_n]+\frac{\lambda_n^2}{\epsilon}\|r_n-Bw_n\|^2 \notag\\
	  &\ \ \ +\epsilon_n^2+\epsilon_n.
	  \end{align}
	
	 We have that $\sum_{n=1}^{\infty}\epsilon_n<\infty$ which implies $\sum_{n=1}^{\infty}\epsilon_n^2<\infty$. Therefore, using the conditions in Theorem \ref{thm1} and Lemma \ref{lm1}, \eqref{eq201} implies that
	  $$ \|x_n-p\| \ \ \text{converges and} \ \|w_n-y_n\| \to 0\ \   \text{a.s}..$$
	  We have \begin{align}
	      \|x_n-y_n\| &\le \|x_n-w_n\|+\|w_n-y_n\| \notag \\
	      & \le \alpha_n \|x_n-x_{n-1}\|+\|w_n-y_n\| \to 0. \label{cvn1}
	  \end{align}
	  Let us set 
\begin{equation}
z_n=(I+\lambda_n A)^{-1}(w_n-\lambda_n Bw_n).
\end{equation}
Then, since $J_{\lambda_n A}$ is nonexpansive, we have 
\begin{equation}\label{e:d5}
\|y_n-z_n\| \le \lambda_n \|Bw_n-r_n\| \to 0\; \text{a.s}..
\end{equation}
Hence \begin{align}
    \|w_n-z_n\|\le\|w_n-y_n\|+\|y_n-z_n\| \to 0 \ \ \text{a.s}..
\end{align}
Let $x^*$ be a weak cluster point of $(x_n)_{n \in \NN}$. 
Then, there exists a subsequence $(x_{n_k})_{k \in \mathbb N} $ which converges weakly to $x^*$ a.s.. By \eqref{cvn1}, $ y_{n_k} \rightharpoonup x^* $ a.s..
It follows from
$y_{n_k} \rightharpoonup x^*$ that $ z_{n_k} \rightharpoonup x^*$.
Since 	$z_{n_k}=(I+\gamma_{n_k}A)^{-1}(w_{n_k}-\gamma_{n_k}B w_{n_k})$, we have
	\begin{equation}
	\frac{ w_{n_k}-z_{n_k}}{\gamma_{n_k}}- B w_{n_k} + B z_{n_k}\in (A+B)z_{n_k}.
	\end{equation}
Since $B$ is $L$-Lipschitz and $(\lambda_n)_{n \in \NN}$ is bounded away from $0$, it follows that 
\begin{equation}
\frac{ w_{n_k}-z_{n_k}}{\gamma_{n_k}}- B w_{n_k} + B z_{n_k} \to 0 \quad{a.s..}
\end{equation}
Using  \cite[Corollary 25.5]{Combet2017}, the sum $A+B$ is  maximally monotone and hence, its graph is closed in $\HH^{weak}\times\HH^{strong}$ \cite[Proposition 20.38]{Combet2017}. 
Therefore, $0\in (A+B) x^*$ a.s., that is $x^*\in \mathcal{P}$ a.s. By Lemma \ref{lm2}, the sequence $(x_n)_{n \in \NN}$ converges weakly to $\bar x \in \mathcal{P}$ a.s. and the proof is completed.

(ii) In case $A+B$ is uniform monotone.\\
We rewrite \eqref{eq20} as
	\begin{align}\label{eq202}
	  \mathsf{E}[\|x_{n+1}-p\|^2|\EuScript{F}_n] &\le (1+\epsilon_n)\|x_n-p\|^2-(1- \lambda_n ^2 L^2) \|w_n-y_n\|^2+ \lambda_n^2 \E[\|s_n-By_n\|^2|\EuScript F_n]\notag \\
	  &\ \ +\lambda_n^2\|r_n-Bw_n\|^2 +2 \lambda_n^2 \scal{By_n-Bw_n}{Bw_n-r_n}-2 \lambda_n\phi(\|y_n-p\|)\notag \\
	  &\ \ +\epsilon_n^2+\epsilon_n.
	  	  \end{align}
	  	  We have 
	  	  \begin{align} \label{ineB2}
	  	     2 \scal{By_n-Bw_n}{Bw_n-r_n} &\le \|By_n-Bw_n\|^2+\|r_n-Bw_n\|^2 \notag\\
	  	     &\le L^2 \|y_n-w_n\|^2+\|r_n-Bw_n\|^2
	  	  \end{align}
Using \eqref{ineB2}, from \eqref{eq202} we have  
	  	  \begin{align} \label{ineE22}
	  	    \mathsf{E}[\|x_{n+1}-p\|^2|\EuScript{F}_n] &\le (1+\epsilon_n)\|x_n-p\|^2-(1- 2\lambda_n ^2 L^2) \|w_n-y_n\|^2+ \lambda_n^2 \E[\|s_n-By_n\|^2|\EuScript F_n]\notag \\
	  &\ \ \ +2\lambda_n^2\|r_n-Bw_n\|^2-2 \lambda_n\phi(\|y_n-p\|)+\epsilon_n^2+\epsilon_n.  
	  	  \end{align}
	  	  From $  \sum_{n \in \NN} \lambda_n^2 <+ \infty$, we derive $\lim \limits_{n \to +\infty} \lambda_n=0$. We have that
	  	  \begin{align}
	  	      \begin{cases}
	  	              \sum_{n \in \NN} \epsilon_n<+\infty\\
	  	              \sum_{n \in \NN} \lambda_n^2 \E[\|s_n-By_n\|^2|\EuScript F_n]<+\infty \ \ a.s.\\
	  	              \sum_{n \in \NN}\lambda_n^2\|r_n-Bw_n\|^2<+\infty \ \ a.s.. 
	  	      \end{cases}
	  	  \end{align}
	  	  Therofore \eqref{ineE22} and Lemma \ref{lm1} imply
	  	  \begin{align}
	  \|x_n-p\| \ \text{converges} \ \ \text{and}\ 	     \sum \limits_{n \in \NN} \lambda_n \phi( \|y_n-p\|)<+\infty \ \ a.s.. 
	  	  \end{align}

Since $  \sum_{n \in \NN} \lambda_n = \infty$, it follows from $ \sum \limits_{n \in \NN} \lambda_n \phi( \|y_n-p\|)<+\infty$ that
$\liminf_{n\to\infty} \phi(\|y_n-p\|)=0$.  Thus, there exists a subsequence $\{n_k\}_{n \in \NN}$ such that $\|y_{n_k}-p\| \to 0$. It follows from \eqref{cvn1} that $\|x_{n_k}-p\| \to 0$ a.s.. Therefore, we infer that $\|x_{n}-p\| \to 0$ a.s.. This completes the proof.
\end{proof}

\begin{remark}\upshape
	With respect to Theorem \ref{thm1}, we observe the following.
	\begin{enumerate}
	\item [{\upshape(1)}]	The conditions in \eqref{dkv} are satisfied if sequences $(\|r_n-Bw_n\|^2)_{n \in \NN}$, $(\mathsf{E}[\|s_n-By_n\|^2|\EuScript{F}_n])_{n \in \NN}$ are bounded.
		\item [{\upshape(2)}] Theorem \ref{thm1} removes the assumption (iii) of Theorem 3.2 in \cite{bang1}, i.e. $\sup_{n \in \NN} \|x_n-x_{n-1}\|^2<+\infty$ and $\sum_{n \in \NN} \alpha_n<+\infty$. 
				\item [{\upshape(3)}] The algorithm (3.2) of \cite{bang6} is a particular case of our algorithm when $\alpha_n=0$. The condition \eqref{e:sum}, i.e. $\sum_{n\in\mathbb{N}}\mathsf{E}[\|s_n-By_n\|^2|\EuScript{F}_n]<
 +\infty \ \text{and}\ \ \sum_{n\in\mathbb{N}}\|r_n-Bw_n\|^2|<+\infty \ \ $ is weaker the conditions $\sum_{n\in\mathbb{N}}\sqrt{\mathsf{E}[\|s_n-By_n\|^2|\EuScript{F}_n]}<
 +\infty\ \  \text{and}\  \sum_{n\in\mathbb{N}}\sqrt{\|r_n-Bw_n\|^2}<+\infty \ \ $ of Theorem 3.1 in \cite{bang6}. 
 \item [{\upshape(4)}] For general case, i.e. $A$ is maximally monotone and $B$ is monotone and Lipschitz, the proposed algoritms in \cite{cui1,cui2} require a large number of samples in each iteration, all are unbiased estimates. However, Algorithm \ref{alg1} only requires $s_n$ is unbiased estimate. The range of the step size $\lambda_n$ in Theorem \ref{thm1} is more extended than that in Theorem 1 of \cite{cui1}. In case $r_n$ and $s_n$ is the average of samples as in \cite{cui1,cui2}, we can obtain the same results as there. 
					\end{enumerate}
\end{remark}
For the rate convergence in uniformly monotone case, we define the function
\begin{align}\label{phi}
    \varphi_c:\ ]0,+\infty[ \to \RR:\ t \mapsto \begin{cases}
            \dfrac{t^c-1}{c} \ &\text{if}\ c \neq 0\\
            \log t \ \  &\text{if}\ c=0.
    \end{cases} 
\end{align}
The following Lemma establishes a non asymptotic bound for numerical sequences satisfying a given recursive inequality. The proof is obtained similarly to the proof of Lemma 3.1 in \cite{bang2}.\\
\begin{lemma} \label{lmsn}
 Let $\alpha \in ]\frac{1}{2},1],\ \beta>1$. Let $a,b \in ]0,+\infty[,\ a \le \beta.$ Set $\alpha_n=\dfrac{a}{n^\alpha},\ \beta_n=\dfrac{b}{n^\beta}$. Let $(s_n)_{n \in \NN}$ be a sequence such that
 \begin{align}
     (\forall n \in \NN)\ \ 0 \le s_{n+1} \le (1-\alpha_n)s_n+\beta_n.
     \end{align}
   Let $n_0$ be the smallest integer such that, for every $n \ge n_0>1,$ it holds $\alpha_n<1,$ set $t=1-2^{\alpha-1} \ge 0.$ Then for every $n \ge 2n_0$, the followings hold
   \begin{enumerate}
       \item If $\alpha=1$, we get 
       \begin{align} \label{alpha11}
    s_{n+1} \le s_{n_0}\big(\dfrac{n_0}{n+1}\big)^a+\dfrac{b}{(n+1)^a}\big(1+\frac{1}{n_0}\big)^a \varphi_{a+1-\beta}(n).
\end{align}
\item If $1/2<\alpha<1$, we have
\begin{align}\label{alpha01}
    s_{n+1} \le \bigg(b\varphi_{1-\beta}(n)+s_{n_0} \text{exp}\bigg(\frac{an_0^{1-\alpha}}{1-\alpha}\bigg)\bigg)\text{exp}\bigg(\frac{-at(n+1)^{1-\alpha}}{1-\alpha}\bigg)+\dfrac{b 2^{\beta-\alpha}}{a  (n-2)^{\beta-\alpha}}.
\end{align}
   \end{enumerate}
\end{lemma}
\begin{proof}
We recall the definition of $\varphi_c$ in \eqref{phi}. Note that, $\varphi_c$ is a increasing function and for $\delta \ge 0$, $2 \le m \le n$, we get:
\begin{align}
    \varphi_{1-\delta}(n+1)-\varphi_{1-\delta}(m) \le \sum_{k=m}^n k^{-\delta} \le \varphi_{1-\delta}(n) . \label{ineint}
\end{align}
We have 
\begin{align}\label{ines}
    s_{n+1} \le s_{n_0}\prod_{k=n_0}^n (1-\alpha_n)+\sum_{k=n_0}^n\prod_{i=k+1}^n(1-\alpha_i)\beta_k.
\end{align}
Let us estimate the first term in the right hand side of \eqref{ines}. Using \eqref{ineint} and the inequality $1-x \le e^{-x}\ \forall \ x\in \RR$, we have 
\begin{align}\label{alpha=1}
    \prod_{k=n_0}^n (1-\alpha_n)&=\prod_{k=n_0}^n (1-ak^{-\alpha}) \le e^{-a \sum \limits_{k=n_0}^n k^{-\alpha}} \notag\\
    &\le \begin{cases}
            \big(\dfrac{n_0}{n+1}\big)^a \ &\text{if}\ \alpha=1\\
            \text{exp}\big(\frac{a}{1-\alpha}(n_0^{1-\alpha}-(n+1)^{1-\alpha})\big) \ &\text{if} \ \frac{1}{2}<\alpha <1.
    \end{cases}
\end{align}
To estimate the second term in the right hand side of \eqref{ines}, let us consider firstly the case $\alpha=1$. We have
\begin{align}\label{alpha=12}
    \sum_{k=n_0}^n\prod_{i=k+1}^n(1-\alpha_i)\beta_k &\le \sum_{k=n_0}^n e^{-a \sum \limits_{i=k+1}^n i^{-\alpha}}\beta_k \notag \\
    &\le \sum_{k=n_0}^n \big(\dfrac{k+1}{n+1}\big)^a\dfrac{b}{k^\beta}=\dfrac{b}{(n+1)^a}\sum_{k=n_0}^n (1+\frac{1}{k})^a k^{a-\beta}\notag \\
    &\le \dfrac{b}{(n+1)^a}\big(1+\frac{1}{n_0}\big)^a \varphi_{a+1-\beta}(n).
\end{align}
From \eqref{ines}, using \eqref{alpha=1} and \eqref{alpha=12}, for $\alpha=1$, we get
\begin{align}
    s_{n+1} \le s_{n_0}\big(\dfrac{n_0}{n+1}\big)^a+\dfrac{b}{(n+1)^a}\big(1+\frac{1}{n_0}\big)^a \varphi_{a+1-\beta}(n).
\end{align}
We next estimate the second term in the right hand side of \eqref{ines} in case $1/2<\alpha<1$. Let $m \in \NN$ such that $n_0 \le n/2 \le m+1 \le (n+1)/2.$ We have
\begin{align}
    \sum_{k=n_0}^n&\prod_{i=k+1}^n(1-\alpha_i)\beta_k=\sum_{k=n_0}^m\prod_{i=k+1}^n(1-\alpha_i)\beta_k+\sum_{k=m+1}^n\prod_{i=k+1}^n(1-\alpha_i)\beta_k \notag \\
    &\le \text{exp}\big(-\sum_{i=m+1}^n \alpha_i\big)\sum_{k=n_0}^m \beta_k+\dfrac{b}{a \cdot m^{\beta-\alpha}}\sum_{k=m+1}^n\prod_{i=k+1}^n(1-\alpha_i)\alpha_n \notag \\
    &\le b \cdot \text{exp}\big(-a\sum_{i=m+1}^n i^{-\alpha}\big)\sum_{k=n_0}^m k^{-\beta}+\dfrac{b}{a \cdot m^{\beta-\alpha}}\sum_{k=m+1}^n\big(\prod_{i=k+1}^n(1-\alpha_i)-\prod_{i=k}^n(1-\alpha_i)\big) \notag \\
    &\le b \cdot \text{exp}\big(\frac{a}{1-\alpha}((m+1)^{1-\alpha}-(n+1)^{1-\alpha}) \big)\varphi_{1-\beta}(n)+\dfrac{b 2^{\beta-\alpha}}{a  (n-2)^{\beta-\alpha}} \notag\\
    &\le b \cdot \text{exp}\big(\frac{-at(n+1)^{1-\alpha}}{1-\alpha} \big)\varphi_{1-\beta}(n)+\dfrac{b 2^{\beta-\alpha}}{a  (n-2)^{\beta-\alpha}}. \label{alpha<1}
\end{align}
Combing \eqref{alpha=1} and \eqref{alpha<1}, for $1/2<\alpha<1$, we have
\begin{align}
    s_{n+1} &\le s_{n_0}\text{exp}\big(\frac{a}{1-\alpha}(n_0^{1-\alpha}-(n+1)^{1-\alpha})\big)+b \cdot \text{exp}\big(\frac{-at(n+1)^{1-\alpha}}{1-\alpha} \big)\varphi_{1-\beta}(n)+\dfrac{b 2^{\beta-\alpha}}{a  (n-2)^{\beta-\alpha}} \notag \\
    &\le \bigg(b\varphi_{1-\beta}(n)+s_{n_0} \text{exp}\bigg(\frac{an_0^{1-\alpha}}{1-\alpha}\bigg)\bigg)\text{exp}\bigg(\frac{-at(n+1)^{1-\alpha}}{1-\alpha}\bigg)+\dfrac{b 2^{\beta-\alpha}}{a  (n-2)^{\beta-\alpha}}.
\end{align}
\end{proof}

\begin{theorem}\label{t:2}
Suppose that $\AAA$ or $B$ is $\mu$-strongly monotone. For $\alpha \in ]1/2,1],\ a>0$, define 
\begin{equation}\label{e:step}
  (\forall n \in \NN)\quad  \lambda_n = \frac{4a}{\mu n^\alpha}.
\end{equation}
Suppose that there exist constants $c$ and $\theta>1$ such that 
\begin{equation}
(\forall n\in\NN)\; \mathsf{E}[(2\|r_n-Bw_n\|^2+\|s_n-By_n\|^2)|\EuScript F_n]  \leq c \ \ a.s. 
\end{equation}
 and $\epsilon_n=\mathcal{O} (n^{-\theta})$. Set $\beta=\min \{2\alpha,\theta\}$, assume that $a \le \beta$. Then
 \begin{align} \label{rate}
	  	  	       \E[\|x_n-p\|^2]=\begin{cases} \mathcal{O}(n^{\alpha-\beta}) \ &\text{if}\ 1/2<\alpha<1,\\
	  	  	           \mathcal{O}(n^{-a})+\mathcal{O}(n^{1-\beta}) \ &\text{if}\ \alpha=1,\ a \neq \beta-1,\\
	  	  	              \mathcal{O}(n^{-a})+\mathcal{O}(\frac{\ln n}{n^a}) \ &\text{if}\ \alpha=1,\ a=\beta-1.
	  	  	  	      \end{cases}
	  	  	  	  \end{align}
\end{theorem}
\begin{proof}Using the strong monotonicity of $A+B$, we rewrite \eqref{ineE22}
 \begin{align} \label{ineE222}
	  	    \mathsf{E}[\|x_{n+1}-p\|^2|\EuScript{F}_n] &\le (1+\epsilon_n)\|x_n-p\|^2-(1- 2\lambda_n ^2 L^2) \|w_n-y_n\|^2+ \lambda_n^2 \E[\|s_n-By_n\|^2|\EuScript F_n]\notag \\
	  &\ \ \ +2\lambda_n^2\|r_n-Bw_n\|^2- \lambda_n\mu\|y_n-p\|^2+\epsilon_n^2+\epsilon_n.  
	  	  \end{align}
	  	  Using Cauchy-Schwarz inequality, we have 
	  	  \begin{align}
	  	      \|x_n-p\|^2&\le 2\big(\|x_n-y_n\|^2+\|y_n-p\|^2 \big) \notag \\
	  	      &\le 4\big(\|x_n-w_n\|^2+\|w_n-y_n\|^2 \big)+2\|y_n-p\|^2\notag \\
	  	      & \le 4 \epsilon_n^2+4\|w_n-y_n\|^2+2\|y_n-p\|^2
	  	  \end{align}
	  	  which is equivalent to
	  	  \begin{align}
	  	      \|y_n-p\|^2 &\ge \dfrac{\|x_n-p\|^2}{2}-2\epsilon_n^2-2\|w_n-y_n\|^2.
	  	  \end{align}
	  	  Hence \eqref{ineE222} implies that
	  	  \begin{align} \label{ineE223}
	  	    \mathsf{E}[\|x_{n+1}-p\|^2|\EuScript{F}_n] &\le (1+\epsilon_n-\frac{\lambda_n\mu}{2})\|x_n-p\|^2-(1- 2\lambda_n ^2 L^2-2\lambda_n \mu) \|w_n-y_n\|^2\notag \\
	  &\ \ \ + \lambda_n^2 \E[\|s_n-By_n\|^2|\EuScript F_n] +2\lambda_n^2\|r_n-Bw_n\|^2+2\lambda_n \mu \epsilon_n^2+\epsilon_n^2+\epsilon_n.  
	  	  \end{align}
	  	  We have that there exists $n_0 \in \NN$ such that $\forall n \ge n_0$ 
	  	  \begin{align}
	  	      \begin{cases}
	  	              \epsilon_n \le \frac{\lambda_n\mu}{4},\\
	  	              1-2\lambda_n^2L^2-2\lambda_n\mu \ge 0,\\
	  	              2\lambda_n \mu \epsilon_n^2+\epsilon_n^2+\epsilon_n \le 2 \epsilon_n.
	  	      \end{cases}
	  	  \end{align}
	  	  Therefore \eqref{ineE223} implies that for $n \ge n_0$, we have
	  	  \begin{align}
	  	      \mathsf{E}[\|x_{n+1}-p\|^2] &\le (1-\frac{\lambda_n\mu}{4})\E[\|x_n-p\|^2]
	 +c\lambda_n^2+2\epsilon_n \notag\\
	 &=(1-an^{-\alpha})\E[\|x_n-p\|^2]+\dfrac{16ca^2}{\mu}n^{-2\alpha}+2\epsilon_n
	  	  \end{align}
	  	  From the definition of $\theta, \beta$, there exist $n_1 \in \NN$ and $b>0$ such that $\forall n \ge n_1$, we get
	  	  \begin{align}
	  	       \mathsf{E}[\|x_{n+1}-p\|^2] \le (1-an^{-\alpha})\E[\|x_n-p\|^2]+b n^{-\beta}.
	  	  \end{align}
	  	  Using Lemma \ref{lmsn}, we obtain:
	  	  
	  	  	  	  In case $1/2<\alpha<1$, from \eqref{alpha01}, we have $\mathsf{E}[\|x_n-p\|^2]=\mathcal{O}(n^{\alpha-\beta}).$ 
	  	  	  	  
	  	  	  	  In case $\alpha=1$, from \eqref{alpha11} and \eqref{phi}, we get
	  	  	  	  \begin{align}
	  	  	  	      \E[\|x_n-p\|^2]=\begin{cases}
	  	  	  	             \mathcal{O}(n^{-a})+\mathcal{O}(n^{1-\beta}) \ &\text{if}\ a \neq \beta-1\\
	  	  	  	             \mathcal{O}(n^{-a})+\mathcal{O}(\frac{\ln n}{n^a}) \ &\text{if}\ a=\beta-1
	  	  	  	      \end{cases}
	  	  	  	  \end{align}
which proves the desired result.
\end{proof}
\begin{remark}Here are some remarks.
\begin{enumerate}
\item[{\upshape(1)}] The strong almost sure convergence of the iterates is obtained from the condition \eqref{dkv}.
    \item[{\upshape(2)}] It follows from \eqref{rate} that the best rate $\mathcal{O}(1/n)$ is derived with $\alpha=1$, $\theta \ge 2$ and $a > 1$. This result is similar to the result in Theorem 3.4 (iii) \cite{bang2}. Note that, we do not require $B$ is cocoercive as in \cite{bang2}.
    \item[\upshape(3)] The rate $\mathcal{O}(1/n)$ is faster than the rate $\mathcal{O}(\log n/n)$ in \cite{DB1}.
    \item [\upshape(4)] In \cite{cui1}, the authors proved the linear convergence of $\E\|x_n-p\|^2$. However, as mentioned above, in each iteration, the algorithm requires a large number of samples and the oracle complexity is still $\mathcal{O}(1/\epsilon)$ which is equal to the complexity as in Theorem \ref{t:2}.
    \end{enumerate}
    \end{remark}
From Theorem \ref{thm1} and Theorem \ref{t:2}, we have the following Corollary:
\begin{corollary}\label{co:1}
Let $f\in\Gamma_0(\HH)$ and
 $h\colon\HH\to\RR$ be a convex differentiable function, with $L$-Lipschitz continuous gradient, given by an  expectation form $h(x)= \E_\xi [H(x,\xi)]$. In the expectation,  $\xi$ is a random vector whose probability distribution  is supported on a set  $\Omega \subset \RR^m$, and $H\colon \HH\times\Omega\to \RR$ is convex function with respect to the variable $x$. The problem is to
\begin{equation}\label{e:probapp}
\underset{  x\in \mathcal{H}}{\text{minimize}} \; f(x) +h(x),
\end{equation}
under the following assumptions: 
\begin{enumerate}
\item $\zer(\partial f +\nabla h)\not=\emp$.
\item It is possible to obtain independent and identically distributed (i.i.d.) samples  $(\xi_n)_{n\in\NN}$, $(\xi_n')_{n\in\NN}$ of $\xi$.
\item  $ \E[\nabla H(x,\xi)] = \nabla h(x)$ for $\forall x \in \HH$. 
\end{enumerate}
Let $(\lambda_n)_{n\in\NN}$ be a sequence in $\left]0,\pinf\right[$. Let
$ x_{-1},x_{0}$ be in $\HH$. Define
\begin{equation}  \label{algo:app1}
 (\forall n\in\NN)\quad
\begin{array}{l}
\left\lfloor
\begin{array}{l}
	w_n=x_n+\alpha_n(x_n-x_{n-1}),\\
y_n=\prox_{\lambda_n f}\big(w_n-\lambda_n \nabla H(w_n,\xi_n)\big)\\
x_{n+1} = y_n-\lambda_n\big(\nabla H(y_n,\xi_n')-\nabla H(w_n,\xi_n)\big),
\end{array}
\right.\\[2mm]
\end{array}\\
\end{equation}
where $(\alpha_n)_{n \in \NN}$ is defined as in Algorithm \ref{alg1}. Denote $\EuScript F_n=\sigma(\xi_0,\xi_0',\ldots,\xi_{n-1},\xi_{n-1}',\xi_n)$.
Then, the followings hold.
\begin{enumerate}
    \item If $f$ is $\mu$-strongly monotone ($\mu\in\left]0,+\infty\right[$),
and there exists a   constant $c$ such that 
\begin{equation}
 \E[(\|\nabla H(y_n,\xi_n')-\nabla h(y_n)\|^2+\|\nabla H(w_n,\xi_n)-\nabla h(w_n)\|^2)|\EuScript F_n] \leq c \ \ a.s..
\end{equation}
Assume that $\epsilon_n=\mathcal{O}(n^{-2})$ then for $ \lambda_{n} = \dfrac{8}{\mu n}$ $\forall n \in \NN$, we obtain
\begin{equation}
\E\left[\|\xx_n-\overline{\xx}\|^2\right] = \mathcal{O}(1/n),
\end{equation}
where $\overline{x}$ is the unique solution to \eqref{e:probapp}.
\item If $f$ is not strongly monotone, let $(\lambda_n)_{n\in\NN}$ be a sequence in $\left]\epsilon, \frac{1-\epsilon}{L}\right[$. Assume that 
\begin{equation}
\begin{cases}
\sum_{n\in\mathbb{N}}\E[\nabla H(w_n,\xi_n)-\nabla h(w_n)\|^2|\EuScript F_n]<+\infty\ \  a.s..\\
\sum_{n\in\mathbb{N}}\E[\nabla H(y_n,\xi_n')-\nabla h(y_n)\|^2|\EuScript F_n]<+\infty\ \  a.s..
\end{cases}
 \end{equation}
  Then  $(x_n)_{n \in \NN}$ converges weakly to a random variable $\overline{x}\colon\Omega\to\zer(\partial f+\nabla h)$ a.s..
    \end{enumerate}
\end{corollary}
\begin{proof} The conclusions are followed from Theorem \ref{thm1} $\&$ \ref{t:2} where 
\begin{equation}
    A =\partial f, B =\nabla h,\; \text{and}\; (\forall n\in\NN)\; r_n = \nabla H(w_n,\xi_n), \ s_n=\nabla H(y_n,\xi_n').
\end{equation}
\end{proof}

\section{Saddle point problem}
Now, we study the primal-dual problem which was firstly investigated
in \cite{Combet2}. This typical structured primal-dual framework covers a widely class of convex
optimization problems and it has found many applications to image processing, machine learning
\cite{Combet2,plc7,Pesquet15,Cham16,Ouy13}.
Based on the duality nature of this framework, we design a new stochastic primal-dual splitting method and research the ergodic convergence
of the primal-dual gap.
\begin{problem}\label{app2} Let $\HH$ and $\GG$ be separable real Hilbert spaces. 
 Let $f\in\Gamma_0(\HH)$, $g\in\Gamma_0(\GG)$ and $h\colon\HH\to\RR$ be a convex differentiable function, with $L_h$-Lipschitz continuous gradient, given by an  expectation form $h(x)= \E_\xi [H(x,\xi)]$. In the expectation,  $\xi$ is a random vector whose probability distribution $P$ is supported on a set  $\Omega_p \subset \RR^m$, and $H\colon \HH\times\Omega\to \RR$ is convex function with respect to the variable $x$.
Let $\ell \colon\GG \to \RR$ be a convex differentiable function with $L_\ell$-Lipschitz continuous gradient, and given by an  expectation form $\ell(v)= \E_\xi [L(v,\xi)]$. In the expectation,  $\zeta$ is a random vector whose probability distribution  is supported on a set $\Omega_D \subset \RR^d$, and $L\colon \GG\times\Omega_D\to \RR$ is convex function with respect to the variable $v$. Let $K\colon\HH\to\GG$ be a bounded linear operator. 
The primal problem is to 
\begin{align}
\underset{  x\in \mathcal{H}}{\text{minimize}} \; h(x)+(\ell^*\vuo g)(Kx)+f(x), \label{e:primalap}
\end{align}
and the dual problem is to 
\begin{align}
\underset{  v\in \mathcal{G}}{\text{minimize}} \; (h+f)^*(-K^*v)+g^*(v) +\ell(v), \label{e:dualap}
\end{align}
under the following assumptions: 
\begin{enumerate}
\item There exists a point $(x^\star,v^\star)\in\HH\times\GG$ such that the primal-dual gap function defined by
\begin{align}
G:&\mathcal{H}\times\mathcal{G} \to \mathbb R \cup \{-\infty,+\infty\} \notag\\
&(x,v) \mapsto h(x)+f(x)+\scal{Kx}{v} -g^*(v) -\ell(v) \label{s}
\end{align}
verifies the following condition:
\begin{align}\label{e:saddle}
\big(\forall x\in\HH\big)\big(\forall v\in\GG\big)\;
G(x^\star,v) \leq G(x^\star,v^\star) \leq G(x,v^{\star}).
\end{align}
\item It is possible to obtain independent and identically distributed (i.i.d.) samples  $(\xi_n,\zeta_n)_{n\in\NN}$ and $(\xi_n',\zeta_n')_{n\in\NN}$  of $(\xi,\zeta)$.
\item  $\forall x \in \HH, \ v \in \GG$, we have
$\E_{(\xi,\zeta)}[(\nabla H(x,\xi),\nabla L(v,\zeta))] = (\nabla h(x),\nabla \ell(v))$. 

\end{enumerate}

\noindent Using the standard technique as in \cite{Combet2}, 
we derive from \eqref{algo:app1} the following
stochastic primal-dual splitting method,
Algorithm \ref{al:spd}, for solving Problem \ref{app2}. The weak almost sure convergence and the convergence in expectation of the resulting algorithm can be derived easily from Corollary \ref{co:1} and hence we omit them here.
\end{problem}

\begin{algorithm}\label{al:spd} 
 Choose $\theta \in [0,1]$. Let $(\lambda_n)_{n \in \NN}$ be a positive sequence, $(\epsilon_n)_{n \in \NN}\subset[0,+\infty)$ satisfying
	 $$\sum_{n=0}^{+\infty}\epsilon_n<+\infty.$$
	 	Let $(x_{-1},x_{0})\in\HH^2$ and  $(v_{-1},v_{0})\in\GG^2$.
	 	Iterates
\begin{equation}\label{e:spd}
\begin{array}{l}
\operatorname{For}\; n=0,1,\ldots,\\
\left\lfloor
\begin{array}{l}
    \alpha_n=\begin{cases}
	\theta&\text{if $x_n=x_{n-1}$,}\\
	\min\bigg\{\theta,\dfrac{\epsilon_n}{\|x_n-x_{n-1}\|}\bigg\}&\text{if $x_n \neq x_{n-1}$.}
	\end{cases}\\
		w_n=x_n+\alpha_n(x_n-x_{n-1})\\
	u_n = v_n+\alpha_n(v_n-v_{n-1})\\
	y_{n}=\text{prox}_{\lambda_n f}(w_n-\lambda_n \nabla H(w_n,\xi_n)  -\lambda_n K^* u_n)\\
	z_{n}=\text{prox}_{\lambda_n g^*}(u_n-\lambda_n\nabla L(u_n,\zeta_n) + \lambda_n Kw_n)\\
	v_{n+1}=z_n-\lambda_n(\nabla L(z_n,\zeta_n')-\nabla L(u_n,\zeta_n))+\lambda_n K(y_n-w_n)\\
	x_{n+1}=y_n-\lambda_n(\nabla H(y_n,\xi_n')-\nabla H(w_n,\xi_n))-\lambda_n K^*(z_n-u_n)
	\end{array}
\right.\\[2mm]
\end{array}
	\end{equation}
\end{algorithm}
For simple, set $\mu = \max\{L_h,L_\ell\}$ and let us define some notations in the space $\HH\times\GG$ where the scalar product and the associated norm are defined in the normal manner,
\begin{equation}
    \begin{cases}
            \mathsf{x} &= (x,v),\quad\; \ww_n=(w_n,u_n), \quad
        \mathsf{x}_n = (x_n,v_n),\quad\;
                \mathsf{y}_n = (y_n,z_n),\; \\
                \rr_n&=(\nabla H(w_n,\xi_n),\nabla L(u_n,\zeta_n)),\\
          \mathsf{s}_n &= (\nabla H(y_n,\xi_n'),\nabla L(z_n,\zeta_n') ),\\
           \nabla \hh(\ww_n)&=(\nabla h(w_n),\ell(u_n)),\ \nabla \hh(\yy_n) = (\nabla h(y_n), \nabla \ell(z_n)).\\
    \end{cases}
\end{equation}

\begin{theorem}\label{thrego} Let $(\lambda_n)_{n\in\NN}$ be a sequence in $\left]0,\frac{1}{\sqrt{1+\epsilon}(\mu+\|K\|)} \right[$ $(\epsilon>0)$ such that
\begin{equation}\label{e:sumgam}
    \; C=\lambda_n^2\E\|\sss_n-\nabla \hh(\yy_n)\|^2 +(1+\frac{1}{\epsilon})\lambda_n^2\E\|\rr_n-\nabla \hh(\ww_n)\|^2 < \infty.
\end{equation} 
For every $N\in\NN$, define 
\begin{equation}
	\hat{y}_N= \bigg(\sum_{n=0}^N\lambda_n y_{n}\bigg)/\bigg(\sum_{n=0}^N\lambda_n\bigg)\; \text{and}\;
	\hat {z}_N=\bigg(\sum_{n=0}^N \lambda_nz_{n}\bigg)/\bigg(\sum_{n=0}^N\lambda_n\bigg).
	\end{equation}
Then the following holds:
\begin{equation}
\E[G(\hat{y}_{N},v)-G(x,\hat{z}_{N})] \leq \bigg(\frac12(1+S T )(\|(x_0,v_0)-(x,v)\|^2+C)\bigg)\bigg/\bigg(\sum_{n=0}^N\lambda_n\bigg),
\end{equation}
where 
$S=\sum \limits_{n \in \NN} \epsilon_n$ and $T=\prod_{n=0}^{+\infty}(1+\epsilon_n).$
\end{theorem}

\begin{proof}
Since $\ell$ is convex, we get
\begin{equation}\label{e:rr1}
    \ell(z_n)\leq \ell(v) +
    \scal{ \nabla \ell(z_n)}{z_n-v}.
\end{equation}
From \eqref{e:spd}, we have 
\begin{equation}
-(z_n-u_n+\lambda_n \nabla L(u_n,\zeta_n)-\lambda_n K w_n) \in
\lambda_n\partial g^*(z_n),
\end{equation}
and hence, using the convexity of $g^*$, 
\begin{equation}
    g^*(v)-g^*(z_n)\geq
    \frac{1}{\lambda_n}\scal{z_n-v}{z_n-u_n+\lambda_n \nabla L(u_n,\zeta_n)-\lambda_n K w_n}.\label{e:rr2}
\end{equation}
Therefore, we derive from \eqref{e:rr1}, \eqref{e:rr2} and \eqref{s} that 
\begin{align}
    G(y_{n},v)&-G(y_{n},z_{n})=\scal{K y_{n}}{v-z_{n}}  -g^*(v)+g^*(z_{n})-\ell(v)+\ell(z_{n}) \notag \\
& \le \scal{ K y_{n}}{v-z_{n}}+\dfrac{1}{\lambda_n }\scal{v-z_{n}}{z_{n}-u_n+\lambda_n \nabla L(u_n,\zeta_n)-\lambda_n K w_n } \notag \\ 
&\ \ +\scal{ \nabla \ell(z_n)}{z_{n}-v} \notag \\
&=\scal {K(y_{n}-w_n)}{v-z_{n}} +\dfrac{1}{\lambda_n} \scal{v-z_{n}}{z_{n}-u_n} \notag \\
&\ \  + \scal{ \nabla \ell(z_n)-\nabla L(u_n,\zeta_n)}{z_{n}-v} \notag \\
&=\scal{K(y_n-w_n)-\nabla L(z_n,\zeta_{n+1/2})+\nabla L(u_n,\zeta_n)}{v-z_n}+\dfrac{1}{\lambda_n} \scal{v-z_{n}}{z_{n}-u_n}\notag\\
&\ \ +\scal{\nabla l(z_n)-\nabla L(z_n,\zeta_{n+1/2})}{z_n-v}\notag \\
&=\scal{\frac{v_{n+1}-z_n}{\lambda_n}}{v-z_n}+\dfrac{1}{\lambda_n} \scal{v-z_{n}}{z_{n}-u_n}\notag\\
&\ \ +\scal{\nabla l(z_n)-\nabla L(z_n,\zeta_{n+1/2})}{z_n-v} \label{ine1d}
\end{align}
By the same way, we have 
\begin{equation}
   h(y_n)-h(x)
   \leq \scal{ \nabla h(y_n)}{y_n-x},
\end{equation}
and
\begin{equation}
   -( y_n-w_n+\lambda_n \nabla H(w_n,\xi_n)+\lambda_n K^*(u_n))
    \in\lambda_n\partial f(y_n),
\end{equation}
which implies
\begin{equation}
    f(y_n)-f(x)\leq \frac{1}{\lambda_n} \scal{x-y_n}{y_n-w_n+\lambda_n \nabla H(w_n,\xi_n)+\lambda_n K^*(u_n)}.
\end{equation}
In turn, we have
\begin{align}
G(y_n,z_n)&-G(x,z_n)=h(y_n)-h(x)+\scal{ K(y_n-x)}{z_n}+f(y_n)-f(x) \notag \\
&\le \scal{ \nabla h(y_n)}{y_n-x} +\scal{K(y_n-x)}{z_n} \notag\\
&\ \ +\frac{1}{\lambda_n} \scal{x-y_n}{y_n-w_n+\lambda_n \nabla H(w_n,\xi_n)+\lambda_n K^*(u_n)} \notag\\
&=\frac{1}{\lambda_n} \scal{x-y_n}{y_n-w_n}+\scal{x-y_n}{\nabla H(w_m,\xi_n)+K^*(u_n-z_n)-\nabla H(y_n,\xi_{n+1/2})}\notag\\
&\ \ +\scal{\nabla h(y_n)-\nabla H(y_n,\xi_{n+1/2})}{y_n-x}\notag\\
&=\frac{1}{\lambda_n} \scal{x-y_n}{y_n-w_n}+\scal{x-y_n}{\dfrac{x_{n+1}-y_n}{\lambda_n}}
\notag \\
&\ \ +\scal{\nabla h(y_n)-\nabla H(y_n,\xi_{n+1/2})}{y_n-x}. \label{ine2d}
\end{align}

It follows from \eqref{ine1d} and \eqref{ine2d} that
\begin{align}
&G(y_n,v)-G(x,z_n) \notag \\
&\quad \le \dfrac{1}{\lambda_n} \big(\scal{\xx-\yy_n}{\yy_n-\ww_n}+\scal{\xx-\yy_n}{\xx_{n+1}-\yy_n} \big)+\scal{\nabla \hh(\yy_n)-\sss_n}{\yy_n-\xx},
\end{align}
which is equivalent to
\begin{align}
    2 \lambda_n& \big(G(y_n,v)-G(x,z_n)\big) \notag \\
    &\le \|\ww_n-\xx\|^2-\|\yy_n-\xx\|^2-\|\yy_n-\ww_n\|^2-\big(\|\xx_{n+1}-\xx\|^2-\|\xx_{n+1}-\yy_n\|^2-\|\yy_n-\xx\|^2 \big)\notag \\
    &\ \ +\scal{\nabla \hh(\yy_n)-\sss_n}{\yy_n-\xx}\notag \\
    &=\|\ww_n-\xx\|^2-\|\xx_{n+1}-\xx\|^2 -\|\yy_n-\ww_n\|^2+\|\xx_{n+1}-\yy_n\|^2 +\scal{\nabla \hh(\yy_n)-\sss_n}{\yy_n-\xx},
 \label{g}
\end{align}
which implies
\begin{align}
    2 \lambda_n\E \big[G(y_n,v)-G(x,z_n)\big|\EuScript F_n] &\le \|\ww_n-\xx\|^2-\E[\|\xx_{n+1}-\xx\|^2|\EuScript F_n]-\|\yy_n-\ww_n\|^2\notag \\
    &\ \ +\E[\|\xx_{n+1}-\yy_n\|^2|\EuScript F_n],
\end{align}
where $\EuScript F_n=\sigma(\xi_0,\zeta_0,\xi_0',\zeta_0',\ldots,\xi_{n-1},\zeta_{n-1},\xi_{n-1}',\zeta_{n-1}',\xi_n,\zeta_n).$\\
Taking expectation both side of above inequality, we get
\begin{align}
    2\lambda_n \E\big[G(y_n,v)-G(x,z_n)\big] &\le \E\|\ww_n-\xx\|^2-\E\|\xx_{n+1}-\xx\|^2-\E\|\yy_n-\ww_n\|^2
    +\E\|\xx_{n+1}-\yy_n\|^2.\label{ineEno}
\end{align}
Now, we estimate the first and the last term in the right side of \eqref{ineEno}. For the first term, using \eqref{hqal}, we have
\begin{align}\label{inewx}
    \|\ww_n-\xx\|^2&=\|\xx_n+\alpha_n(\xx_n-\xx_{n-1})-\xx\|^2=\|\xx_n-\xx\|^2+\alpha_n^2\|\xx_n-\xx_{n-1}\|^2\notag \\
    &\ \ +2\alpha_n\scal{\xx_n-\xx}{\xx_n-\xx_{n-1}} \notag\\
    &\le \|\xx_n-\xx\|^2+\epsilon_n^2+2\epsilon_n\|\xx_n-\xx\|\notag\\
    &\le (1+\epsilon_n)\|\xx_n-\xx\|^2+\epsilon_n^2+\epsilon_n.
\end{align}
For the last term of \eqref{ineEno}
\begin{align}
    \|\xx_{n+1}-\yy_n\|^2=\|x_{n+1}-y_n\|^2+\|v_{n+1}-z_n\|^2.
\end{align}
From \eqref{e:spd}, we get
\begin{align}
    \|x_{n+1}-y_n\|^2&=\|\lambda_n(\nabla H(y_n,\xi_n')-\nabla H(w_n,\xi_n))+\lambda_n K^*(z_n-u_n)\|^2 \notag \\
    &=\lambda_n^2\|(\nabla H(y_n,\xi_n')-\nabla h(y_n))+\nabla h(y_n)-\nabla H(w_n,\xi_n))+K^*(z_n-u_n)\|^2 \notag \\
    &=\lambda_n^2\bigg(\|\nabla H(y_n,\xi_n')-\nabla h(y_n)\|^2+\|\nabla h(y_n)-\nabla H(w_n,\xi_n))+K^*(z_n-u_n)\|^2 \notag \\
    &\ \ \ +2 \scal{\nabla H(y_n,\xi_n')-\nabla h(y_n)}{\nabla h(y_n)-\nabla H(w_n,\xi_n))+K^*(z_n-u_n)}\bigg),
\end{align}
which implies that 
\begin{align}
    \E\|x_{n+1}-y_n\|^2 &\le \lambda_n^2\E[\|\nabla H(y_n,\xi_n')-\nabla h(y_n)\|^2] \notag \\
    &\ +\lambda_n^2\E[\|\nabla h(w_n)-\nabla H(w_n,\xi_n))+\nabla h(y_n)-\nabla h(w_n)+K^*(z_n-u_n)\|^2] \notag \\
    &\le \lambda_n^2\E[\|\nabla H(y_n,\xi_n')-\nabla h(y_n)\|^2]+(1+\frac{1}{\epsilon})\lambda_n^2\E[\|\nabla h(w_n)-\nabla H(w_n,\xi_n)\|^2]\notag \\
    &\ +(1+\epsilon)\lambda_n^2\E[\|\nabla h(y_n)-\nabla h(w_n)+K^*(z_n-u_n)\|^2]. \label{ineE1}
\end{align}
The last term in the right side of \eqref{ineE1} can be bounded
\begin{align}
    \|\nabla h(y_n)-\nabla h(w_n)+K^*(z_n-u_n)\|^2&=\|\nabla h(y_n)-\nabla h(w_n)\|^2+\|K^*(z_n-u_n)\|^2 \notag \\
    &\ +2\scal{\nabla h(y_n)-\nabla h(w_n)}{K^*(z_n-u_n)}\notag \\
    &\le L_h^2\|y_n-w_n\|^2+\|K\|^2\|z_n-u_n\|^2\notag \\
    &\ +2L_h \|K\|\|y_n-w_n\|\|z_n-u_n\|. \label{ineE2}
\end{align}
From \eqref{ineE1} and \eqref{ineE2}, we get
\begin{align}
    \E\|x_{n+1}&-y_n\|^2\le \lambda_n^2\E[\|\nabla H(y_n,\xi_n')-\nabla h(y_n)\|^2]+(1+\frac{1}{\epsilon})\lambda_n^2\E[\|\nabla h(w_n)-\nabla H(w_n,\xi_n)\|^2]\notag \\
    &\ +(1+\epsilon)\lambda_n^2\E [ L_h^2\|y_n-w_n\|^2+\|K\|^2\|z_n-u_n\|^2+2L_h \|K\|\|y_n-w_n\|\|z_n-u_n\|].\label{ineE3}
\end{align}
Similarly to \eqref{ineE3}, we have
\begin{align}
    \E\|v_{n+1}&-z_n\|^2\le \lambda_n^2\E[\|\nabla L(z_n,\zeta_n')-\nabla \ell(z_n)\|^2]+(1+\frac{1}{\epsilon})\lambda_n^2\E[\|\nabla \ell(u_n)-\nabla L(u_n,\zeta_n)\|^2]\notag \\
    &\ +(1+\epsilon)\lambda_n^2\E [ L_\ell^2\|z_n-u_n\|^2+\|K\|^2\|y_n-w_n\|^2+2L_\ell \|K\|\|z_n-u_n\|\|y_n-w_n\|].\label{ineE4}
\end{align}
Combining \eqref{ineE3} and \eqref{ineE4}, we derive
\begin{align}
    \E\|\xx_{n+1}-\yy_n\|^2&\le \lambda_n^2\E\|\sss_n-\nabla \hh(\yy_n)\|^2+(1+\frac{1}{\epsilon})\lambda_n^2\E\|\rr_n-\nabla \hh(\ww_n)\|^2\notag \\
    &\ +(1+\epsilon)(\mu+\|K\|)^2\lambda_n^2\E\|\yy_n-\ww_n\|^2\label{ineEx}
\end{align}
From \eqref{ineEno}, \eqref{inewx} and \eqref{ineEx}, using the condition of $\lambda_n$, i.e. $\lambda_n \le \frac{1}{\sqrt{1+\epsilon}(\mu+\|K\|)} $, we get
\begin{align}
2\lambda_n\E \big[G(y_n,v)-G(x,z_n)\big]&\le (1+\epsilon_n)\E\|\xx_n-\xx\|^2-\E\|\xx_{n+1}-\xx\|^2+c_n,\label{ineEE}
\end{align}
where $c_n=\lambda_n^2\E\|\sss_n-\nabla \hh(\yy_n)\|^2 +(1+\frac{1}{\epsilon})\lambda_n^2\E\|\rr_n-\nabla \hh(\ww_n)\|^2$. It follows from \eqref{ineEE} that
\begin{align}
    \E\|\xx_{n+1}-\xx\|^2\le (1+\epsilon_n)\E\|\xx_n-\xx\|^2+c_n.
\end{align}
Using above inequality $n$ times, we obtain
\begin{align}
     \E\|\xx_{n+1}-\xx\|^2&\le \prod_{i=0}^n(1+\epsilon_i)\|\xx_0-\xx\|^2+\sum \limits_{i=0}^n \prod_{j=i+1}^n(1+\epsilon_{j})c_i\notag\\
     &\le T \|\xx_0-\xx\|^2+T \sum \limits_{n \in \NN} c_n=T \|\xx_0-\xx\|^2+T C.
\end{align}
Summing \eqref{ineEE} from $n=0$ to $n=N$, we have
\begin{align}
    2\sum \limits_{n=0}^N \lambda_n\E \big[G(y_n,v)-G(x,z_n)\big]&\le \|\xx_0-\xx\|^2+\sum \limits_{n=0}^{N} \epsilon_n \E\|\xx_n-\xx\|^2-\E\|\xx_{n+1}-\xx\|^2+\sum\limits_{n=0}^Nc_n\notag \\
    &\le \|\xx_0-\xx\|^2+S(T \|\xx_0-\xx\|^2+T C)+C\notag \\
    &=(1+T S)(\|\xx_0-\xx\|^2+C).
\end{align}
Using the convexity-concavity of $G$, we obtain the desired result.
\end{proof}
\begin{remark}Here are some remarks.
\begin{enumerate}
    \item [\upshape(1)] The upper bound of $(\lambda_n)_{n \in \NN}$ in Theorem \ref{thrego} can be written as $\dfrac{1-\epsilon'}{\mu+\|K\|}$, where $\epsilon' \in ]0,1[.$ Indeed, for $\epsilon<\frac{1}{(1-\epsilon')^2}-1$, we get $\lambda_n<\frac{1}{\sqrt{1+\epsilon}(\mu+\|K\|)}.$ 
    \item [\upshape(2)] In the deterministic setting, 
    the convergence rate of the primal-dual gap for structure convex optimization involving infimal  convolutions were also investigated in \cite{Bot1,Bot2}. Furthermore, in the deterministic case and $\alpha_n=0$ $\forall n \in \NN$, our result is one in \cite{Bot2}. The convergence rate $\mathcal{O}(1/ N)$ of the primal-dual gap also obtain in \cite{te15}. 
    \item [\upshape(3)] In the stochastic setting, our  result is the same convergence rate of the primal-dual gap as in \cite{bang4,DB1}.
\end{enumerate}

\end{remark}

\bibliographystyle{abbrv}

\bibliography{AllBiD}

\end{document}